\documentclass[11pt]{article}

\usepackage{amsmath, amssymb, amsthm, amstext, graphicx,tikz, url}
\usepackage{enumerate}
\usepackage{fullpage}

\newtheorem{theorem}{Theorem}[section]
\newtheorem{lemma}[theorem]{Lemma}

\newtheorem{definition}[theorem]{Definition}
\newtheorem{corollary}[theorem]{Corollary}

\newcommand{\into}{\rightarrow}

\newcommand{\st}{\mathrel{:}}

\newcommand{\R}{\mathbb{R}}
\newcommand{\Erdos}{Erd\H{o}s }
\newcommand{\Erdosns}{Erd\H{o}s}

\makeatletter
\def\blfootnote{\xdef\@thefnmark{}\@footnotetext}
\makeatother

\begin{document}

\title{Distinct Volume Subsets via Indiscernibles}
\author{
{William Gasarch}
\thanks{
Department of Computer Science,
       University of Maryland at College Park, College Park, MD 20742.
Email: \texttt{gasarch@cs.umd.edu}.
}
\and
{Douglas Ulrich}
\thanks{
Department of Mathematics,
        University of Maryland at College Park, College Park, MD 20742.
Email: \texttt{ds\_ulrich@hotmail.com}.
The author was partially supported by NSF Research Grant DMS-1308546.}
}

\blfootnote{2010 \emph{Mathematics Subject Classification:} 03E75, 03C98}

	\blfootnote{\emph{Key Words and Phrases:} Ramsey theory, distinct volumes, indiscernibles}

\date{\today}

\maketitle

\begin{abstract}
\Erdos proved that for every infinite $X \subseteq \mathbb{R}^d$ there is $Y \subseteq X$ with $|Y|=|X|$, such that all pairs of points from $Y$ have distinct distances, and he gave partial results for general $a$-ary volume. In this paper, we search for the strongest possible canonization results for $a$-ary volume, making use of general model-theoretic machinery. The main difficulty is for singular cardinals; to handle this case we prove the following. Suppose $T$ is a stable theory, $\Delta$ is a finite set of formulas of $T$, $M \models T$, and $X$ is an infinite subset of $M$. Then there is $Y \subseteq X$ with $|Y| = |X|$ and an equivalence relation $E$ on $Y$ with infinitely many classes, each class infinite, such that $Y$ is $(\Delta, E)$-indiscernible. We also consider the definable version of these problems, for example we assume $X \subseteq \mathbb{R}^d$ is perfect (in the topological sense) and we find some perfect $Y \subseteq X$ with all distances distinct. Finally we show that \Erdosns's theorem requires some use of the axiom of choice.
\end{abstract}

\section{Introduction}

In this paper we use the term 1-ary volume for length, 2-ary volume for area,
3-ary volume for volume. We may use the term volume when the dimension is
understood. Also, the natural number $n$ is identified with the set $\{0, \ldots, n-1\}$.

A set $X\subseteq \R^d$ is {\it $a$-rainbow} if all  
$a$-sets of points that yield nonzero volumes
have distinct volumes. Let $h_{a,d}(n)$ be the largest integer 
$t$ such that any set of $n$ points in $\R^d$ 
contains a rainbow subset of $t$.  
This function was studied by Conlon et. al~\cite{billdist}
which also includes references to past work.

In this paper, we are interested in the case where the cardinality of the set of points 
is some $\kappa$ with
$\aleph_0\le \kappa \le 2^{\aleph_0}$. \Erdos was the first to consider this in \cite{remarks}. Using different terminology, he proved the following:

\begin{theorem}\label{ErdosTheorem}
If $X \subseteq \mathbb{R}^d$ is infinite then there is a $2$-rainbow $Y \subseteq X$ with $|Y| = |X|$.
\end{theorem}
The case when $|X|$ is countable can be dealt with quickly using the canonical Ramsey theorem of \Erdos and Rado \cite{CanRamsey}. Alternatively, it is equivalent to apply Ramsey's theorem to a coloring $g: \binom{X}{2a} \to c$, for $c < \omega$ large enough. Namely, given $s \in \binom{X}{2a}$, define $g(s)$ so as to encode the set of all pairs $(u, v)$ from $\binom{s}{a}$ having the same volume. For our purposes, we find this latter approach more natural, although some of what we do could be phrased in the language of the canonical Ramsey theorem. 

 \Erdosns's proof of Theorem~\ref{ErdosTheorem} is complicated by the possibility that $|X|$ is singular. He notes the following holds by an easier proof:

\begin{theorem}\label{ErdosTheorem2}
If $X \subseteq \mathbb{R}^d$ is infinite with $|X|$ regular, and $2 \leq a \leq d+1$, then there is an $a$-rainbow $Y \subseteq X$ with $|Y| = |X|$.
\end{theorem}

\Erdos also gives the following example:

\begin{theorem}\label{ErdosExample}
If $\lambda \leq 2^{\aleph_0}$ is singular then there is $X \subseteq \mathbb{R}^d$ of size $\lambda$, such that there is no $3$-rainbow $Y \subseteq X$ with $|Y| = \lambda$.
\end{theorem}
\begin{proof}
Write $\mbox{cof}(\kappa)= \lambda$. Let $(\ell_\alpha: \alpha < \lambda)$ be $\lambda$-many parallel lines in $\mathbb{R}^2$. Let $(\kappa_\alpha: \alpha < \lambda)$ be a cofinal sequence of regular cardinals in $\kappa$. Choose $X_\alpha \subseteq \ell_\alpha$ of cardinality $\kappa_\alpha$ and let $X = \bigcup_{\alpha < \lambda} X_\alpha$. Let $Y \subseteq X$ have cardinality $\kappa$. We claim that $Y$ cannot be $3$-rainbow. Indeed, write $Y_\alpha = Y \cap X_\alpha = Y \cap \ell_\alpha$. Then there must be cofinally many $\alpha < \lambda$ with $Y_\alpha$ infinite, as otherwise $|Y| \leq \kappa_\alpha + \lambda < \lambda$ for some $\alpha < \lambda$. Thus we can find $\alpha < \beta < \lambda$ such that $Y_\alpha$ and $Y_\beta$ are both infinite. Let $v_0, v_1$ be two distinct points in $Y_\alpha$, and let $w_0, w_1$ be two distinct points in $Y_\beta$. Then the triangles $(v_0, v_1, w_0)$ and $(v_0, v_1, w_1)$ have the same nonzero area.
\end{proof}

We are interested in strengthenings and generalizations of Theorem~\ref{ErdosTheorem2} for uncountable sets. We will give stronger canonization results than just $a$-rainbow. 
Namely, say that 
$X$ is strongly $a$-rainbow if all $a$-subsets of $X$ yield distinct, 
nonzero volumes, and say that $X$ is strictly $a$-rainbow if $X$ 
is strongly $a'$-rainbow for all $a' \leq a$, and $X$ is a subset of 
an $a-1$-dimensional hyperplane. 
(In particular, all $a+1$-subsets of $X$ have volume $0$.) As an example, if $a_{n, i}: n < \omega, i < d$ are algebraically independent reals, and if we set $\overline{a}_n = (a_{n, 0}, \ldots, a_{n, d-1}) \in \mathbb{R}^d$, then $X := \{\overline{a}_n: n < \omega\}$ is strongly $d+1$-rainbow, and thus strictly $d+1$-rainbow. Moreover, if $\rho: \mathbb{R}^{d} \to \mathbb{R}^{d'}$ is any isometric embedding, then the image of $X$ under $\rho$ is also strictly $d+1$-rainbow.
%
%\noindent
%{\bf Examples:}
%\begin{enumerate}
%\item
%Let $d\ge 1$.
%Let $X\subseteq \R^d$ be countable. By the canonical Ramsey theorem
%applied to the coloring that maps a pair of points in $X$ to its 2-volume
%(normally called length)
%we obtain a countable strongly 2-rainbow set $Y$.
%Note that two points of $Y$ cannot have a zero volume.
%.
%
%\item
%Let $d\ge 2$.
%Let $X\subseteq \R^d$ be countable. By the canonical Ramsey theorem
%applied to the coloring that maps a triple of points in $X$ to its 3-volume
%(normally called area) 
%we obtain either (1) a countable strongly 3-rainbow set $Y$, or (2) a countable set
%$X'$ of points on a line, so all of the 3-volumes are 0.
%In either case we can apply the canonical Ramsey theorem
%to the coloring that maps a pair of points in $X'$ to its 2-volume to
%obtain a strongly 2-rainbow set $Z \subseteq Y$.
%\item Generally, given $X \subseteq \mathbb{R}^d$ countably infinite, we can apply the canonical Ramsey theorem for each $2 \leq a \leq d+1$ to obtain $Y \subseteq X$ countably infinite, and such that $Y$ is $a$-rainbow for each $2 \leq a \leq d+1$. Then for each $2 \leq a \leq d+1$, if $Y$ has a single nondegenerate $a$-subset then they are all nondegenerate, and so we get that if $a$ is largest such that some or any $a$-subset of $Y$ is nondegenerate, then $Y$ is strictly $a$-rainbow.
%
%\end{enumerate}

In Section~\ref{ModelTheory}, we begin by reviewing some model-theoretic results of Shelah (Theorems~\ref{IndiscAleph0},~\ref{IndiscRegular},~\ref{IndiscRegular2}), dealing with the following situation: we are given $T$ stable, $M \models T$, and $X \subseteq M$ infinite, and we try to find $Y \subseteq X$ with $|Y| = |X|$ and $Y$ indiscernible. These theorems only deal with the case where $|X|$ is regular; Theorem~\ref{ErdosExample} above shows that obstacles exist for the singular case. The problem is the presence of an equivalence relation $E$ on $X$ that divides $X$ into fewer than $\kappa$-many classes, each of size less than $\kappa$. Theorems~\ref{IndiscSingularUncountCof} and ~\ref{IndiscSingularCountCof} together demonstrate that this is the only obstruction, using a weakened notion of indiscernibility with respect to an equivalence relation $E$. We remark that the combinatorial argument for Theorem~\ref{IndiscSingularCountCof} has other applications; we give a purely finitary analogue in Theorem~\ref{Finitary}.

In Section~\ref{kappaRainbow}, we consider $X \subseteq \mathbb{R}^d$ of size $\kappa$ for some uncountable cardinal $\kappa$, and we try to get $Y \subseteq X$ of size $\kappa$ which is as nice as possible with respect to $a$-ary volumes for $a \leq d+1$, using the results of Section~\ref{ModelTheory}. We prove in Theorem~\ref{RainbowRegular} that for every regular cardinal $\kappa \leq 2^{\aleph_0}$, and for every $X \subseteq \mathbb{R}^d$ of size $\kappa$, there is some $X' \subseteq X$ of size $\kappa$ and some $2 \leq a \leq d+1$ such that $X$ is strictly $a$-rainbow. We proceed as follows: given $X \subseteq \mathbb{R}^d$ of size $\kappa$, we obtain a sufficiently indiscernible $Y \subseteq X$ using Theorem~\ref{IndiscRegular2}, using the stability of $(\mathbb{C}, +, \cdot, 0, 1)$. Then we apply geometric arguments to argue that $Y$ is strictly $a$-rainbow for some $a$. We note that it is possible to prove Theorem~\ref{RainbowRegular} directly, similarly to Theorem~\ref{ErdosTheorem2}. 

For singular cardinals, we know from Theorem~\ref{IndiscSingularCountCof} that there is some finite list of possible configurations, although we cannot identify it explictly. We are at least able to give some information about what the configurations look like in Theorem~\ref{SingularRainbow}; in particular, they are all $2$-rainbow, and so we recover \Erdosns's Theorem~\ref{ErdosTheorem}.

In Section~\ref{PSPSection}, we consider what happens for $X \subseteq \mathbb{R}^d$ which is reasonably definable. Our main result is Theorem~\ref{th:perfect}: if $P \subseteq \mathbb{R}^d$ is perfect, then there is a perfect $Q \subseteq P$ and some $a \leq d+1$ such that $Q$ is strictly $a$-rainbow. Our main tool is a Ramsey-theoretic result of Blass \cite{Blass} concerning colorings of perfect trees.

In Section~\ref{independence}, we show it is independent of ZF whether or not every uncountable subset of $\mathbb{R}$ has an uncountable $2$-rainbow subset.

In this paper we work in ZFC, with the exception of Section~\ref{PSPSection}, which is in ZF + DC.

\section{Some remarks on indiscernibles}\label{ModelTheory}
We first review the notion of local indiscernibility, following Shelah \cite{ShelahIso}. $T$ will always be a complete first order theory in a countable language.

Suppose $\Delta$ is a collection of formulas of $T$, $M \models T$ and $A \subseteq M$. Given a finite tuple $\overline{b}$ from $M$, define $tp_{\Delta}(\overline{b}/A)$ to be the set of all formulas $\phi(\overline{x}, \overline{a})$ such that $\overline{a} \in A^{<\omega}$ and $\phi(\overline{x}, \overline{y}) \in \Delta$ and $M \models \phi(\overline{b}, \overline{a})$. 

Suppose also that $I$ is an index set, and $(\overline{a}_i: i \in I)$ is a sequence from $M^d$ for some $d < \omega$. Then:

\begin{itemize}
\item  We say that $(\overline{a}_i: i \in I)$ is $\Delta$-\emph{indiscernible} over $A$ if: given $i_0, \ldots, i_{n-1}$ all distinct elements of $I$, and given $j_0, \ldots, j_{n-1}$ also distinct elements from $I$, then $tp_\Delta(\overline{a}_{i_0}, \ldots, \overline{a}_{i_{n-1}}/A) = tp_{\Delta}(\overline{b}_{i_0}, \ldots, \overline{b}_{i_{n-1}})$. In this case the indexing doesn't matter and so we also say that $\{\overline{a}_i: i \in I\}$ is \emph{indiscernible} over $A$.
\item If $I$ is linearly ordered, then we say that $(\overline{a}_i: i \in I)$ is $\Delta$-\emph{order-indiscernible} over $A$ if: for every $i_0 < \ldots < i_{n-1}$, $j_0 < \ldots < j_{n-1}$ from $X$, $tp_\Delta(\overline{a}_{i_0}, \ldots, \overline{a}_{i_{n-1}}/A) = tp_{\Delta}(\overline{b}_{i_0}, \ldots, \overline{b}_{i_{n-1}})$.
\end{itemize}
When we do not mention $A$, we mean $A = \emptyset$.

The following is an easy application of Ramsey's theorem  (as recorded for instance in Lemma 2.3 of Chapter I of \cite{ShelahIso}:

\begin{theorem}\label{IndiscAleph0Order}
Suppose $T$ is a complete first order theory in a countable language, and $\Delta$ is a finite collection of formulas of $T$. Suppose $M \models T$, and $(\overline{a}_n: n < \omega)$ is an infinite sequence from $M^d$. Then there is some infinite subsequence $(\overline{a}_n: n \in I)$ which is $\Delta$-order indiscernible.
\end{theorem}

We will be mainly interested in the case where $T$ is stable. In this case, the following is part of Theorem 2.13 of Chapter II of \cite{ShelahIso}:

\begin{theorem}\label{IndiscAleph0}
Suppose $T$ is a stable complete first order theory in a countable language, and $\Delta$ is a finite collection of formulas of $T$. Suppose $M \models T$, and $(\overline{a}_n: n < \omega)$ is an infinite sequence from $M^d$. Then $(\overline{a}_n: n < \omega)$ is order-indiscernible if and only if it is indiscernible (in fact this characterizes stability). Hence, if $X \subseteq M^d$ is infinite, then there is an infinite, $\Delta$-indiscernible $Y \subseteq X$.  
\end{theorem}

We are interested in generalizations of Theorem~\ref{IndiscAleph0} to the case where $X$ has uncountable cardinality $\kappa$.  Shelah has proved several results along these lines for regular cardinals; we give two versions. The first requires $T$ to be $\omega$-stable, and gets full indiscernibility. See Remark 2 after 
Theorem 2.8 from Chapter 1 of \cite{ShelahIso}.

\begin{theorem}\label{IndiscRegular}
Let $T$ be an $\omega$-stable theory (we can suppose in a countable language). Let $\kappa$ be a regular uncountable cardinal and let $d < \omega$. Then whenever $M \models T$, $A \subseteq M$ has size less than $\kappa$, and $X \subseteq M^d$ has size $\geq \kappa$, we can find some finite sequence $\overline{a} \in M$, and we can find some stationary type $p(\overline{x}) \in S^d(\overline{a})$, such that there is some $Y \subseteq X$ of size $\kappa$ which is a set of independent realizations of $p(\overline{x}) | A \overline{a}$. In particular, $Y$ is indiscernible over $A$.

\end{theorem}
\begin{proof}
For the reader's convenience we provide a proof.

We can suppose $T = T^{eq}$, and thus that we can code finite tuples as single elements. Also we can suppose that $|X| = \kappa$. Enumerate $X = (a_\alpha: \alpha < \kappa)$. For each $\alpha < \kappa$, write $X_\alpha = \mbox{acl}\left(A \cup \{a_\beta: \beta < \alpha\}\right)$, and choose a formula $\phi_\alpha(x)$ over $X_\alpha$ of the same Morley rank as $tp(a_\alpha/X_\alpha)$, and of Morley degree $1$. By Fodor's lemma, we can find $S \subseteq \kappa$ stationary such that $\phi_\alpha(x) = \phi_\beta(x) = \phi(x)$ for all $\alpha, \beta \in S$. Choose $\alpha_*$ large enough that $\phi(x)$ is over $X_{\alpha_*}$. Write $\phi(x) = \phi(x, a)$ for some $a \in X_\alpha$. Let $p(x)$ be the unique type over $a$ containing $\phi(x, a)$ and of the same Morley rank; then for all $\alpha_* \leq \alpha \in S$, $tp(a_\alpha/ X_\alpha)$ is the unique non-forking extension of $p(x)$ to $X_\alpha$. From this it follows easily that $Y := \{a_\alpha: \alpha \in S \backslash \alpha_*\}$ is as desired.
\end{proof}

The second version applies to any stable theory, but only gives local indiscernibility. It is Theorem 2.19 of Chapter II of \cite{ShelahIso}, and it strictly generalizes the final claim of Theorem~\ref{IndiscAleph0}.

\begin{theorem}\label{IndiscRegular2}
Let $T$ be a stable theory and let $\Delta$ be a finite set of formulas. Let $\kappa$ be a regular  cardinal and let $d < \omega$. Then whenever $M \models T$, $A \subseteq M$ has size less than $\kappa$, and $X \subseteq M^d$  has size $\geq \kappa$, then there is some $Y \subseteq X$ of size $\kappa$ such that $Y$ is $\Delta$-indiscernible over $A$.
\end{theorem}

When $\kappa$ is singular, the na\"{i}ve generalization of Theorem~\ref{IndiscRegular} fails, by Theorem~\ref{ErdosExample}. (In fact, one can easily modify this example to work in the theory of equality.) The problem here is the presence of an equivalence relation on $X$ that has fewer than $\kappa$ classes, with each class of size less than $\kappa$, and the behavior of elements in distinct classes differs from the behavior of elements in the same class. In fact, we show this is the only obstruction.

We wish to formalize the notion of ``indiscernible up to an equivalence relation." This a special case of generalized indiscernibles, introduced by Shelah in \cite{ShelahIso} Section VII.2, and further analyzed (with slightly varying definitions) in several subsequent papers, e.g. \cite{genIndiscEx}. (In these works, the focus is on using these generalized indiscernibles to build Ehrenfeucht-Mostowski models; our interest is different, in that we want to extract generalized indiscernibles from a given $X$.)

Suppose $T$ is a complete first order theory, and $\Delta$ is a collection of formulas. Suppose $M \models T$, and $A \subseteq M$, and $X \subseteq M^d$ for some $d$. Finally suppose $E$ is an equivalence relation on $X$. Then $X$ is $(\Delta, E)$-\emph{indiscernible} over $A$ if for every $\overline{a}_0, \ldots, \overline{a}_{n-1}$, $\overline{b}_0, \ldots, \overline{b}_{n-1}$ sequences from $X$ with each $\overline{a}_i \not= \overline{a}_j$ and each $\overline{b}_i \not= \overline{b}_j$, \emph{if}  for every $i < j < n$, $\overline{a}_i E \overline{a}_j$ if and only if $\overline{b}_i E \overline{b}_j$, \emph{then} $tp_\Delta(\overline{a}_i: i < n/A) = tp_\Delta(\overline{b}_i: i < n/A)$.  $X$ is $E$-indiscernible over $A$ if $X$ is $(\Delta, E)$-indiscernible over $A$, where $\Delta$ is the collection of all formulas of $T$.

So for instance, if $X$ is $\Delta$-indiscernible, then letting $E_=$ be the equivalence relation of identity on $X$, we have that $X$ is $(\Delta, E_=)$-indiscernible; and also $X$ is $(\Delta, X \times X)$-indiscernible.

We have the following adaptations of Theorem~\ref{IndiscRegular} for singular cardinals $\kappa$; again, we have two versions. In the first version, we need $T$ to be $\omega$-stable and we need $\kappa$ to have uncountable cofinality, and for this we get full $E$-indiscernibility. In the second version, we just need $T$ to be stable and $\kappa$ can be arbitrary, but for this we just get local $E$-indiscernibility.
\begin{theorem}\label{IndiscSingularUncountCof}
Let $T$ be an $\omega$-stable theory, and let $\kappa$ be a singular cardinal of cofinality $\lambda > \aleph_0$, and let $d < \omega$. Then whenever $M \models T$, $A \subseteq M$ has size less than $\mbox{cof}(\kappa)$, and $X \subseteq M^d$ has size $\geq \kappa$, there is some $Y \subseteq X$ of size $\kappa$ and some equivalence relation $E$ on $Y$, such that $E$ has at most $\lambda$-many equivalence classes, with each equivalence class infinite, and such that $Y$ is $E $-indiscernible over $A$.
\end{theorem}

\begin{proof}
We can suppose $T = T^{eq}$, and thus that $d = 1$.

Write $X$ as the disjoint union of $X_\alpha: \alpha < \lambda$, where each $|X_\alpha| = \kappa_\alpha < \kappa$ is a successor cardinal bigger than $|A|$, and $\kappa_\alpha < \kappa_\beta$ whenever $\alpha < \beta$. By applying Theorem~\ref{IndiscRegular} to each $X_\alpha$ and then pruning, we can suppose there is some $a_\alpha \in M$ and some stationary $p(x) \in S^1(a_\alpha)$, such that $Y_\alpha$ is an independent set of realizations of $p(x) | \left(A \cup \bigcup_{\beta < \alpha} X_\beta \right)$. Define the equivalence relation $E$ on $X$ by: $a  E b$ iff $a, b$ are in the same $X_\alpha$. 

For each $\alpha < \lambda$, choose $\phi_\alpha(x, a) \in p_\alpha(x)$ of the same Morley rank as $p(x)$, and of Morley degree $1$. By further pruning, we can suppose $\phi_\alpha(x, y) = \phi_\beta(x, y)$ for all $\alpha, \beta < \lambda$. 

Apply Theorem~\ref{IndiscRegular} and prune to get some $a_*$ and some stationary type $q(x) \in S^1(a_*)$, such that $\{a_\alpha: \alpha < \lambda\}$ is an independent set of realizations of $q(x) | A a_*$.

We claim now that $X$ is $E$-indiscernible. To check this, it is convenient to discard all but countably many elements of each $X_\alpha$. Thus enumerate each $X_\alpha = \{b^\alpha_n: n < \omega\}$. Now each $(b^\alpha_n: n < \omega)$ is a Morley sequence in $p_\alpha(x) | \left(A, a_\alpha, \bigcup_{\beta < \alpha} X_\beta \right)$, and $(a_\alpha: \alpha < \lambda)$ is a Morley sequence in $q(x) | A a_*$. It follows by typical nonforking arguments that for each $\alpha$, $X_\alpha$ is indiscernible over $A \cup \bigcup_{\beta \not= \alpha} X_\beta$; also, for each permutation $\sigma$ of $\lambda$, the permutation $\sigma_*$ of $X$ defined by $\sigma_*(b^\alpha_n) = b^{\sigma(\alpha)}_n$ is partial elementary over $A$. From these two facts it is easy to check that $X$ is $E$-indiscernible.
\end{proof}

Before the following theorem, we need some notation. Given $X$ a set and $r < \omega$, $\binom{X}{r}$ denotes the $r$-element subsets of $X$. When $X$ is a set of ordinals, each $s \in \binom{X}{r}$ has a canonical increasing enumeration; thus we can identify $\binom{X}{r} \subseteq X^r$.

\begin{theorem}\label{IndiscSingularCountCof}
Let $T$ be a stable theory, and let $\Delta$ be a finite collection of formulas of $T$. Let $\kappa$ be a singular cardinal, and let $d < \omega$. Then whenever $M \models T$, $A \subseteq M$ has size less than $\mbox{cof}(\kappa)$, and $X \subseteq M^d$ has size $\geq \kappa$, there is some $Y \subseteq X$ of size $\kappa$ and some equivalence relation $E$ on $Y$, such that $E$ has $\aleph_0$-many classes, with each equivalence class infinite, and such that $Y$ is $(\Delta, E)$-indiscernible over $A$. 
\end{theorem}
\begin{proof}
Again, we can suppose $T =  T^{eq}$ and $d = 1$. Let $r$ be the maximum of the arities of formulas of $\Delta$.

Write $\lambda= \mbox{cof}(\kappa)$ and write $X$ as the disjoint union of $X_\alpha: \alpha < \lambda$, where each $|X_\alpha| = \kappa_\alpha < \kappa$ is a successor cardinal bigger than $|A|$, and $\kappa_{\alpha} < \kappa_{\alpha'}$ whenever $k < k'$. 

Then by many applications of Theorem~\ref{IndiscRegular}, we can find distinct $a^{i}_{\alpha, j}: \alpha < \lambda, i, j < r$ such that each $a^i_{\alpha, j} \in X_\alpha$, and we can find $Y_\alpha^i: \alpha < \omega, i \leq r$, such that:

\begin{itemize}
\item Each $Y_\alpha^{i+1} \subseteq Y_\alpha^i \subseteq X_\alpha$;
\item Each $Y_\alpha^i$ has size $\kappa_k$;
\item Each $a^{i}_{\alpha, j} \in Y_\alpha^i \backslash Y_\alpha^{i+1}$;
\item Each $Y_\alpha^i$ is indiscernible over $A \cup \bigcup_{\beta < \alpha} Y_\beta^i \cup \{a^{i'}_{\beta, j}: \beta < \lambda, i' < i, j < r\}$.
\end{itemize} 
For each $\alpha< \lambda$, let $\mathbf{b}_\alpha = (a_{\alpha, j}^i: i, j < r)$. By applying Theorem~\ref{IndiscRegular2}, we can suppose that $(\mathbf{b}_\alpha: \alpha < \lambda)$ is $\Delta$-indiscernible over $A$. Given an injective $s: r \to \lambda$, let $p_s(x_{i, j}: i, j < r) = tp_\Delta(a_{s(i), j}^{r-1-i}: i, j < r)$. This is a $\Delta$-type of a finite tuple over the empty set; in particular it is equivalent to a formula, but it is more convenient to write it as a type. Note that by $\Delta$-indiscernibility, for all $s, s'$, $p_s(x_{i, j}: i , j < r) = p_{s'}(x_{i, j}: i, j < r)$, so we can drop the subscripts and refer to just $p(x_{i, j}: i, j<r)$.

For each $\alpha$, write $Y_\alpha = Y_\alpha^{r}$ and write $Y = \bigcup_\alpha Y_\alpha$. We claim that $Y$ is $(\Delta, E \restriction_Y)$-indiscernible. 

Note that for all $s \in \binom{\omega}{r}$ and for all distinct $a_{s(i), j}: i, j < r$, with $a_{s(i), j} \in Y_{s(i)}$, we have that  $M \models p(a_{s(i), j}: i, j < r)$. This is because, starting from $(a_{s(r-1), j}: j<r)$ and moving downwards, we can shift each $(a_{s(i), j}: j < r)$ to $(a^{r-1-i}_{s(i), j}: j<r)$; when moving $(a_{s(i), j}: j < r)$ to $(a^{r-1-i}_{s(i), j}: j < r)$ we are using the indiscernibility hypothesis on $Y_{s(i)}^{r-1-i}$.   

We now need to take care of the fact that we are only looking at the increasing enumeration of $s$ in the above.

Choose distinct $(a_{\alpha, j}: \alpha < \lambda, j < r)$ with each $a_{\alpha, j} \in Y_k$. Write $\mathbf{a}_\alpha = (a_{\alpha, j}: j < r)$. By the preceding, we have that $(\mathbf{a}_\alpha: \alpha < \lambda)$ is order-$\Delta$-indiscernible; but by Theorem~\ref{IndiscAleph0}, this implies that $(\mathbf{a}_\alpha: \alpha < \lambda)$ is fully indiscernible. In particular, given any injective sequence $s: r \to \lambda$, and given distinct $a_{s(i), j}: i, j < r$, with $a_{s(i), j} \in Y_{s(i)}$, we have that $M \models p(a_{s(i), j}: i, j < r)$ (since we could have chosen $(\mathbf{a}_\alpha: \alpha < \lambda)$ to cover $\mbox{range}(s)$). From this it follows easily that $Y$ is $(\Delta, E \restriction_Y)$-indiscernible. 
\end{proof}

The following theorem is an easy consequence of Theorems~\ref{IndiscRegular2} and~\ref{IndiscSingularCountCof}.

\begin{theorem}\label{FiniteBasis}
Suppose $T$ is stable, $M \models T$, $\Delta$ is a finite set of formulas,  $d< \omega$, and $\kappa$ is an infinite cardinal. Then there is a finite list $(C_i: i < i_*)$ such that: each $C_i \subseteq M^d$ has size $\kappa$, and for every $X \subseteq M^d$ of size $\kappa$, there is some $Y \subseteq X$ of size $\kappa$ and some $i < i_*$ such that $tp_{\Delta}(C_i) = tp_{\Delta}(Y)$ (i.e. there is a bijection $f: C_i \to Y$ that preserves $\Delta$-formulas). If $\kappa$ is regular, then each $C_i$ is $\Delta$-indiscernible; otherwise, each $C_i$ is $(\Delta, E_i)$-indiscernible for some equivalence relation $E_i$ on $C_i$ with infinitely many classes, each class infinite.
\end{theorem}

We give a purely finitary analogue of Theorem~\ref{IndiscSingularCountCof}. First, given $n, r, c < \omega$, an equivalence relation $E$ on $n$ and a function $f: \binom{n}{r} \to c$, say that $X \subseteq n$ is $E$-homogeneous for $f$ if for all $s, t \in \binom{X}{r}$, if $s(i) E s(j)$ iff $t(i) E t(j)$ for all $i, j < r$, then $f(s) = f(t)$. Also, given $A \subseteq Y \subseteq n$, say that $A$ is convex in $Y$ if whenever $n_0 < n_1 < n_2 < n$, if $n_0, n_2 \in A$ and $n_1 \in Y$ then $n_1 \in A$.

Theorem~\ref{Finitary} also follows from the Claim (proved below) and the fact that convexly ordered equivalence relations form a Ramsey class, see \cite{RamseyClasses}. On the other hand, the given proof of Theorem~\ref{Finitary} extends easily to handle infinite cardinals, using the \Erdos-Rado theorem in place of Ramsey's theorem.

\begin{theorem}\label{Finitary}
Suppose $K, L, r, c < \omega$ are given. Then there are $K_*, L_* < \omega$ large enough so that whenever $n \geq K_* \cdot L_*$, and whenever $E$ is an equivalence relation on $n$ with at least $K_*$ many classes, of size at least $L_*$, and whenever $f: \binom{n}{r} \to c$, there is some $X \subseteq n$ which is $E$-homogeneous for $f$ such that $E \restriction_X$ has at least $K$ many classes, each convex in $X$ and of size at least $L$.
\end{theorem}

First, we want the following claim.

\vspace{1 mm}
\noindent \textbf{Claim.} Suppose $K, L < \omega$ are given. Then there are $K_*, L_* < \omega$ large enough so that whenever $n \geq K_* \cdot L_*$, and whenever $E$ is an equivalence relation on $n$ with at least $K_*$ many classes, each of size at least $L_*$, there is some $X \subseteq n$ such that $E \restriction_X$ has at least $K$ many classes, each convex in $X$ and of size at least $L$.
\begin{proof}
Choose $K_*$ such that $K_* \rightarrow (K)^2_{2}$. Choose $L_*$ such that $L_* \rightarrow (L^2)^2_{(2K_*)!}$.

Suppose $E$ is given. We can suppose $E$ is an equivalence relation on $n$ with exactly $K_*$-many classes, each of size exactly $L_*$; so $n = K_* L_*$. Let $(X_k: k < K_*)$ list the equivalence classes of $E$ in some order. For each $k < K_*$, $\ell < L_*$, let $X_k(\ell)$ denote the $\ell$'th element of $X_k$ (listed in increasing order). Define a map $f: \binom{L_*}{2} \to (2K_*)!$, where $f(\ell_0, \ell_1)$ codes the ordering of the elements $(X_{k}(\ell_i): k < K_*, i < 2)$. Choose $I \subseteq L_*$ of size $L^2$ which is homogeneous for $f$. 

Then we have the following: suppose $k_0, k_1 < K_*$. Then one of the following holds, after possibly switching $k_0$ and $k_1$: either $X_{k_0}(\ell_0) < X_{k_1}(\ell_1)$ for all $\ell_0, \ell_1 \in I$; or else $X_{k_0}(\ell_0) < X_{k_1}(\ell_0) < X_{k_0}(\ell_1)$ for all $\ell_0 < \ell_1$ both in $I$. Define $g: \binom{K_*}{2} \to 2$ so that $g(\{k_0, k_1\})$ says which of these cases $\{k_0, k_1\}$ is in (say $0$ for the first case, $1$ for the second case). Choose $J \subseteq K_*$ homogeneous for $g$ of size $K$. Let $X = \{X_k(\ell): k \in J, \ell \in I\}$. 

If $J$ is homogeneous of color $0$, then the classes on $X \restriction_E$ are already convex in $X$ and so we are done. Otherwise, given equivalence classes $Y, Y'$ of $E \restriction_X$, say that $Y <_* Y'$ if $Y(\ell) < Y'(\ell)$ for some or any $\ell < L^2$; let $(Y_k: k < K)$ list the equivalence classes of $E \restriction_X$ in $<_*$-increasing order. Note that for all $\ell_0 < \ell_1 < L^2$, and for all $k_0 < k_1 < K$, $Y_{k_0}(\ell_0) < Y_{k_1}(\ell_0) < Y_{k_0}(\ell_1) < Y_{k_1}(\ell_1)$. Let $Z_k = \{Y_k(Lk), \ldots, Y_k(L(k+1) - 1)\}$. Then clearly $Z= \bigcup_{k < K} Z_k$ works.
\end{proof}

Thus, to prove Theorem~\ref{Finitary}, it suffices to restrict to equivalence relations $E$ such that each class is convex in $n$. Note then that whenever $X \subseteq n$, each class of $E \restriction_X$ will be convex in $X$.

We will need some notation. Given a function $f: \binom{n}{r} \to c$, and given a set of parameters $A \subseteq n$, say that $X \subseteq n$ is homogeneous for $f$ over $A$ if for every $s \in \binom{A}{< r}$, the induced function $f_s: \binom{n}{r-|s|} \to c$ is constant on $X$. For the purposes of this theorem, say that $n \rightarrow (m)^{r}_{c, p}$ if: whenever $f: \binom{n}{r} \to c$, and whenever $A \subseteq n$ has size at most $p$, there is $X \subseteq n$ of size $m$ which is homogeneous for $f$ over $A$. Easily, if $n \rightarrow (m)^{r}_{c'}$, where $c' = c^{p^r 2^r}$, then $n \rightarrow (m)^{r}_{c, p}$.

We now are ready to prove Theorem~\ref{Finitary}.

\begin{proof}
We follow the proof of Theorem 2.5. Choose $K_*$ so that $K_* \rightarrow (K+r)^{r}_{c'}$, where $c' = c^{r^{2r}}$. Choose numbers $(L_{k}^i: k < K_*, -1 \leq i \leq r)$ such that:
\begin{itemize}
\item For all $0 \leq i \leq r$ and for all $k < K_*$, $L_{k}^{i-1} \rightarrow (L_{k}^{i})^{r}_{c, c'}$, where $c' = i\cdot r \cdot K_* + \sum_{k' < k} L_{k'}^{i}$.
\item  Each $L^r_k = L$.
\end{itemize}
Set $L_* = L_{0}^{-1}$ (which we can suppose is the maximum of $(L_{k}^{-1}: k < K_*)$).  Then we claim this works. 

Indeed, suppose $E$ is an equivalence relation on $N$ with at least $K_*$ many classes, each convex of size at least $L_*$, and suppose $f: \binom{n}{r} \to c$. Let $X_k: k < k_*$ list in increasing order the first $K_*$-many classes of $E$. By choice of $(L^i_k: -1 \leq i \leq r, k < K_*)$, we can find $Y^i_k: 0 \leq i \leq r, k < K_*$ and distinct $a^i_{k, j}: 0 \leq i, j < r, k < K_*$ such that:

\begin{itemize}
\item For each $0 \leq i < r$ and $k < K_*$, $Y^{i+1}_k \subseteq Y^{i}_k \subseteq X_k$;
\item For each $0 \leq i \leq r$ and $k < K_*$, $|Y^i_k| = L^i_k$;
\item For each $0 \leq i < r$, $k < K_*$ and $j < r$, $a^{i}_{k, j} \in Y^i_k \backslash Y^{i+1}_k$;
\item Suppose $0 \leq i \leq r, k < K_*$; set $A = \bigcup_{k' < k} Y^{i}_{k'} \cup \{a^{i'}_{k', j'}: 0 \leq i' < i, k' < K_*, j' < r\}$. Then $Y^{i}_k$ is homogeneous for $f$ over $A$.
\end{itemize}

Given $\overline{a} = (a_i: i \in I)$ an injective sequence from $n$, by $\mbox{tp}(s)$ we mean the function $\binom{I}{r} \to c$ induced from $f: \binom{\{a_i: i \in I\}}{r} \to c$.

Write $c' = c^{r^{2r}}$ (as in the definition of $K_*$) and choose $g: \binom{K_*}{r} \to c'$ so that for all $s \in \binom{K_*}{r}$, $g(s)$ codes $tp(a^{r-1-i}_{s(i), j}: i, j < r)$. By choice of $K_*$, we can find $I' \subseteq K_*$ of size $K+r$, which is homogeneous for $g$. Let $I$ be the first $K$-many elements of $I'$. Write $Y = \bigcup_{k \in I} Y^r_k$. Then we claim $Y$ is $E$-homogeneous for $f$.

Indeed, suppose $s, t \in \binom{Y}{r}$, such that for all $i, j < r$, $s(i) E s(j)$ iff $t(i) E t(j)$. We want to show that $f(s) = f(t)$.  Write $s$ as the disjoint union of its equivalence classes listed in increasing order: $s = \bigcup_{i < i_*} s_i$, and similarly $t = \bigcup_{i < i_*} t_i$. Note each $|s_i| = |t_i| = r_i$, say. For each $i < i_*$, let $k_i \in I$ be such that $s_i \subseteq Y_{k_i}^r$, and let $k'_i \in I$ be such that $t_i \subseteq Y_{k'_i}^r$.

Let $s'_i = \{a^{i_*-1-i}_{k_i, j}: j < r_i\}$ and let $t'_i = \{a^{i_*-1-1}_{k'_i, j}: j < r_i\}$. Let $s' = \bigcup_{i < i_*} s'_i$ and let $t' = \bigcup_{i < i_*} t'_i$. Note that by choice of $I$ and $I'$, we have that $f(s') = f(t')$. (We can choose $u \in \binom{I'}{r}$ such that $\{k_i: i < i_*\}$ are the first $i_*$-many elements of $u$, and $v \in \binom{I'}{r}$ such that $\{k'_i: i < i_*\}$ are the first $i_*$-many elements of $v$. Then apply the definition of $g$.) So by symmetry, it suffices to show that $f(s) = f(s')$.

Starting with $i = i_*-1$, shift each $s_i$ to $s'_i$; at each step, we do not change the value of $f$ by our homogeneity assumption on $Y^{i_* -1-i}_{k_i}$. 
\end{proof}

\section{Getting Large Strictly Rainbow Sets}\label{kappaRainbow}
In this section, we are interested in applying the machinery of the previous section to analyze $a$-ary-volumes of subsets of $\mathbb{R}$. Recall that we are interested in the following kind of problem: given $X \subseteq \mathbb{R}^d$ infinite and given $a \leq d+1$, can we find $Y \subseteq X$ with $|Y| = |X|$, such that all distinct $a$-element sets from $Y$ give distinct volumes?

The most natural structure to work in for this would be $(\mathbb{R},+, \cdot, 0, 1)$, but the first-order theory of this structure is unstable. Thus we view $\mathbb{R} \subseteq \mathbb{C}$ and work in the larger field $(\mathbb{C}, +, \cdot, 0, 1)$ instead; it is well known that its first order theory, $ACF_0$, is $\omega$-stable. We could alternatively look under the hood of Theorem~\ref{IndiscRegular2} and note that it applies to $\mbox{Th}(\mathbb{R}, +, \cdot, 0, 1)$ provided $\Delta$ is taken to be a set of quantifier-free formulas, but this really amounts to the same thing.

First we need to show that the relevant notions of $a$-ary volumes are definable in $\mathbb{C}$.

\begin{definition}\label{CayleyMenger}
Let $a \leq d+1 < \omega$. Note that for $(v_0, \ldots, v_{a-1}) \in (\mathbb{R}^d)^{a+1}$, the $a-1$-ary volume of $(v_0, \ldots, v_{a-1})$ is $\frac{1}C |q_a(v_0, \ldots, v_{a-1})|$ for some constant $C$ and some polynomial $q_a(v_0, \ldots, v_{a-1})$ (namely the Cayley-Menger determinant). Thus $(v_0, \ldots, v_{a-1})$ has $a-1$-ary volume $0$ iff $q_a(v_0, \ldots, v_{a-1}) = 0$. Let $p_a(v_0, \ldots, v_{a-1}, w_0, \ldots, w_{a-1})$ be the polynomial in $2da$ variables given by $(q_a(v_0, \ldots, v_{a-1}) - q_a(w_0, \ldots, w_{a-1}))(q_a(v_0, \ldots, v_{a-1}) +q_a(w_0,\ldots, w_{a-1}))$.  Note that for $(v_0, \ldots, v_{a-1}), (w_0, \ldots, w_{a-1})$ from $(\mathbb{R}^d)^a$, we have that their $a-1$-ary volumes are the same if and only if $p_a(v_0, \ldots, v_{a-1}, w_0, \ldots, w_{a-1}) = 0$.

Let $\Delta_d$ be the finite collection of formulas of $ACF_0$ of the form: $q_a(\overline{v}) = 0$, or $p_a(\overline{v}, \overline{w}) = 0$, for $a \leq d+1$.
\end{definition}

Thus, Theorem~\ref{FiniteBasis} gives, for each infinite cardinal $\kappa$, a finite basis of all possible $\Delta_d$-configurations, and we wish to understand the ones that can be embedded in $\mathbb{R}^d$. In the case where $\kappa$ is regular, we succeed completely with Theorem~\ref{RainbowRegular}. First we need the following lemma.

\begin{lemma}\label{RainbowLemma}
Let $d$ be given. Suppose $X \subseteq \mathbb{R}^d$ is infinite, and when viewed as a subset of $\mathbb{C}^d$, is $\Delta_d$-indiscernible. Then $X$ is strictly $a$-rainbow for some $2 \leq a \leq d+1$.
\end{lemma}
\begin{proof}
Let $a$ be largest so that for some or any $v_0, \ldots, v_{a-1} \in X$, we have that $q_a(v_0, \ldots, v_{a-1}) \not= 0$. Then clearly $X$ is a subset of the $a-1$-dimensional hyperplane spanned by any $a$ elements from $X$, and for every $2 \leq a' \leq a$, every $a'$-subset of $X$ has nondegenerate volume, so it suffices to show that $X$ is $a'$ rainbow for all $2 \leq a' \leq a$. Suppose not; say $(v_0, \ldots, v_{a'-1})$ and $(w_0, \ldots, w_{a'-1})$ are from $\binom{X}{a'}$ of the same volume, that is $q_{a'}(v_0, \ldots, v_{a'}-1) = \pm q_{a'}(w_0, \ldots, w_{a'-1})$. We can suppose there is $\ell < a'-1$ such that $v_i = w_i$ for all $i < \ell$, and $v_i \not= w_j$ for any $i, j \geq \ell$. By $a'-\ell$-many applications of indiscernibility of $X$, we can suppose $\ell = a'-2$, or in other words: for all $v_0, \ldots, v_{a'-1}, v_{a'}$ from $X$ distinct, $q_{a'}(v_0, \ldots, v_{a'-2}, v_{a'-1}) = \pm q_{a'}(v_0, \ldots, v_{a'-2}, v_{a'})$. Let $v_0, \ldots, v_{a'-2}, w_n: n < \omega$ be distinct elements from $X$. For each $n < \omega$ let $V_n \subseteq \mathbb{C}^d$ be the set of all $v$ such that $q_{a'}(v_0, \ldots, v_{a'-2}, v) = \pm q_{a'}(v_0, \ldots, v_{a'-2}, w_{n'})$ for all $n' < n$. Clearly this is a descending sequence of prevarieties, and moreover for each $n$, $w_n \in A_n \backslash A_{n+1}$ so it is strict. This contradicts Hilbert's Basis theorem.
\end{proof}

\begin{theorem}\label{RainbowRegular}
Suppose $\kappa \leq 2^{\aleph_0}$ is a regular cardinal, and $X \subseteq \mathbb{R}^d$ has $|X| = \kappa$. Then there is $Y \subseteq X$ of size $\kappa$, and some $2 \leq a \leq d+1$, such that $Y$ is strictly $a$-rainbow.
\end{theorem}
\begin{proof}
By Theorem~\ref{IndiscRegular2} we can choose $Y \subseteq X$ of size $\kappa$ such that $Y$ is $\Delta_d$-indiscernible; then we conclude by Lemma~\ref{RainbowLemma}.
\end{proof}

For singular cardinals we do not have such an explicit conclusion, although we can say something about what the configurations look like:

\begin{theorem} \label{SingularRainbow}
Suppose $X \subseteq \mathbb{R}^d$, and $E$ is an equivalence relation on $X$ with infinitely many classes, each class infinite. Suppose that considered as a subset of $\mathbb{C}^d$, we have that $X$ is $(\Delta_d, E)$-indiscernible. Then there is $2 \leq a_* \leq d+1$ such that each $E$-equivalence class is strictly $a_*$-rainbow, and $X$ is strongly $a$-rainbow for all $a \leq a_*$.
\end{theorem}
\begin{proof}
There is some $2 \leq a_* d+1$ such that each $E$-equivalence class is strictly $a_*$-rainbow, by Lemma~\ref{RainbowLemma}  and $(\Delta_d, E)$-indiscernibility, so we just need to show that $X$ is strongly $a$-rainbow for all $a \leq a_*$.

Suppose not; say $a \leq a_*$ and $u_0, \ldots, u_{a-1}, v_0, \ldots, v_{a-1}$ are $a$-tuples from $X$ with the same volume (possibly $0$). We can suppose $u_{a-1} \not= v_i$ for any $i < a$. We claim we can arrange that $u_i = v_i$ for all $i < a-1$, and that $u_{a-1} E v_{a-1}$. Indeed, choose $w_0, \ldots, w_{a-1}$ distinct elements of $X$ such that $w_i = u_i$ for all $i < a-1$, and $w_{a-1}$ is some new element with $w_{a-1} E u_{a-1}$. By $(\Delta, E)$-indiscernibility, $\overline{v}$ has the same volume as both $\overline{u}$ and $\overline{w}$, so the latter two have the same volume. So replace $\overline{v}$ by $\overline{w}$.

Let $I \subseteq a$ be the set of all $i < a$ with $u_i E u_a$. Then clearly for any $w_0 \ldots w_{a-1}$, if $w_i = u_i$ for all $i < a$ with $i \not \in I$, and if $w_i E u_i$ for all $i \in I$, then $\overline{u}$ and $\overline{w}$ have the same volume. Moreover this holds whenever we replace $u_0, \ldots, u_{a-1}$ by $u'_0, \ldots, u'_{a-1}$, where $u'_i E u'_j$ iff $u_i E u_j$. By reordering we can suppose $I = \{k, k+1, \ldots, a-1\}$, for some $k < a$. Now $k > 0$, since otherwise $\overline{u}, \overline{v}$ are a subset of a single equivalence class, and so this contradicts the choice of $a_*$.

For the contradiction, we suppose we have arranged to have $\{u_i/E: i < a\}$ of minimum size.

Choose elements $u_i^\ell: i < a, \ell < \omega$ as follows:  having defined $u_i^\ell$ for each $\ell < \ell_*$, let $u_i^{\ell_*}: i < k$ be some new elements such that  $u_i^{\ell_*} E u_j^{\ell_*}$ iff $u_i E u_j$, and for $i < k$, $u_i^{\ell_*}$ is not $E$-related to any previous  $u_j^\ell$ or any $u_j$.

For each $\ell < \omega$ let $A_\ell \subseteq \mathbb{C}^{2\cdot d \cdot (a-k)}$ be the pre-variety of all $(v_k, \ldots, v_{a-1}, w_k, \ldots, w_{a-1})$ such that for all $\ell' < \ell$, $p_a(u^{\ell'}_0, \ldots, u^{\ell'}_{k-1}, v_k, \ldots, v_{a-1}, u^{\ell'}_0, \ldots, u^{\ell'}_{k-1}, w_k, \ldots, w_{a-1}) = 0$ (for tuples in $\mathbb{R}^{d \cdot (a-k)}$ recall this is equivalent to saying that $u^{\ell'}_0, \ldots, u^{\ell'}_{k-1}, v_k, \ldots, v_{a-1}$ and \\$u^{\ell'}_0, \ldots, u^{\ell'}_{k-1}, w_k, \ldots, w_{a-1}$ have the same volume). 

This is a descending chain of pre-varieties; but it must also be strict: for let $\ell < \omega$, and let $v_k, \ldots, v_{a-1}, w_k, \ldots, w_{a-1}$ be new elements with each $v_i, w_j E u^\ell_0$. Since $\{u_i/E: i < a\}$ was chosen of minimal size we must have that  $(v_k, \ldots, v_{a-1}, w_k, \ldots, w_{a-1}) \in A_{\ell} \backslash A_{\ell+1}$. But this contradicts Hilbert's Basis theorem. 
\end{proof}

Note that as a special case we have recovered \Erdosns's Theorem~\ref{ErdosTheorem}: whenever $X \subseteq \mathbb{R}^d$ is infinite, there is a $2$-rainbow $Y \subseteq X$ with $|Y| = |X|$. This is because we must have $a_* \geq 2$.

\section{Perfect subsets of $\mathbb{R}^d$ have rainbow perfect subsets}\label{PSPSection}
This section is in ZF+DC.

In this section we show that if $X\subseteq \R^d$ is perfect, then there is some perfect $Y\subseteq X$ which is
strictly $a$-rainbow for some $a\le d+1$. Since $Y$ is perfect we get
$|Y|=|X|$.

\begin{definition}
We make several definitions.

\begin{enumerate}
\item Suppose $P$ is a Polish space. A coloring $f:P \into [c]$ has the {\it Baire property} if, for all $i\in [i]$, $f^{-1}(i)$ has the Baire property.

\item
If $x,y\in 2^\omega$ then 
$\Delta(x,y)=\min \{ i \st x(i)\ne y(i)\}$.

\item if $u \in \binom{2^\omega}{a}$, writing $u = \{x_1, \ldots, x_a\}$ in lexicographically increasing order (as always), then say that $u$ is \emph{skew} if for all $1 \leq i < j < a$, $\Delta(x_i, x_{i+1}) \not= \Delta(x_j, x_{j+1})$. Let $\binom{2^\omega}{a}_{\mbox{skew}}$ be the set of all $u \in \binom{2^\omega}{a}$ which are skew.
\item
We define $f_*: \binom{2^\omega}{a}_{\mbox{skew}} \to LO([a-1])$ as follows, where $LO([a-1])$ is the set of linear orders of $a-1$ (of which there are $(a-1)!$). Namely let $f_*(x_0,\ldots,x_{a-1})$ be the linear ordering $<$ of $a-1$ given by: $i < j$ iff $\Delta(x_i, x_{i+1}) < \Delta(x_j, x_{j+1})$. Here we are writing $u = \{x_0, \ldots, x_{a-1}\}$ in increasing lexicographic order.

\item A perfect subtree of $2^{<\omega}$ is a subtree $T$ of $2^{<\omega}$ (nonempty and closed under initial segments) such that for every $s \in T$ there are $t_0, t_1 \in T$ with $s \subset t_0, t_1$ and such that $t_0$ and $t_1$ are incompatible. Note that the set of branches $[T]$ through $T$ is a perfect subset of $2^\omega$, and this characterizes the perfect subsets of $2^\omega$. Perfect subsets of $2^\omega$ are also called Cantor sets. We say that a Cantor set $C$ is \emph{skew} if for all $x \not= y$, $x' \not= y'$ elements from $C$, if $\Delta(x, y) = \Delta(x', y')$ then $\{x, y\} = \{x', y'\}$. In particular $\binom{C}{a} \subseteq \binom{2^\omega}{a}_{\mbox{skew}}$ for each $a$.
\end{enumerate}
\end{definition}

The following is due to Blass \cite{Blass}. 

\begin{theorem}\label{CanRamseyForCantorSpace}
Suppose $a, c$ are natural numbers and $f:\binom{2^\omega}{a} \into [c]$ has the Baire property. Then there exists a skew Cantor set $C \subseteq 2^\omega$ so that for all $u, v \in \binom{C}{a}$, if $f_*(u) = f_*(v)$ then $f(u) = f(v)$. Thus $f \restriction_{\binom{C}{a}}$ takes on only $(a-1)!$ values, and in fact there are only $c^{(a-1)!}$ possibilities for $f \restriction_{\binom{C}{a}}$.
\end{theorem}

A set $X \subseteq \mathbb{R}^d$ has the perfect set property if $X$ is either countable or else has a perfect subset. This is a regularity property of subsets of $\mathbb{R}^d$, and so holds for all reasonably definable subsets. For instance, every analytic set has the 
perfect set property: see Theorem 12.2 of \cite{Kanamori}. Also, assuming 
sufficient large cardinals (for instance, infinitely many Woodin cardinals with a measurable cardinal above),
all projective subsets of $\mathbb{R}^d$ have this property; see Theorem 32.14 of \cite{Kanamori}.

Moreover, if we let $PSP$ denote the assertion that every subset of $\mathbb{R}^d$ has the perfect set property, then $ZF + DC+ PSP$ is consistent relative to an inaccessible cardinal (this is part of Solovay's theorem, see Theorem 11.11 of \cite{Kanamori}).

\begin{theorem}\label{th:perfect}
Suppose $P \subseteq \mathbb{R}^d$ is perfect. 
Then there is a perfect set $Q \subseteq X$ and some 
$2 \leq a \leq d+1$ such that $Q$ is strictly $a$-rainbow.
\end{theorem}
\begin{proof}
 It is not hard to find a continuous injection $\rho: 2^\omega \to P$, such that the image of $\rho$ is closed. Note that if $C \subseteq 2^\omega$ is a Cantor set then $\rho[C]$ is perfect.

Define $f: \binom{2^\omega}{2a} \to [c]$ (for $c$ large) so that $f(u)$ codes the following information: for each $a' \leq a$, and for each $I, J \in \binom{2a}{a'}$, whether or not the volume of $\rho[u_I]$ is equal to $0$, and whether or not the volume of $\rho[u_I]$ is equal to the volume of $\rho[u_J]$. Clearly we can choose $f$ to have the Baire property (in fact, its graph will be Borel).

Let $C \subseteq 2^\omega$ be a skew Cantor set as in Theorem~\ref{CanRamseyForCantorSpace}. It suffices to show that there some $a \leq d+1$ such that $\rho[C]$ is strictly $a$-rainbow.

Suppose for some $a \leq d+1$ and for some $x_0<_{lex} \ldots <_{lex} x_{a-1}$ from $C$, the volume of $(\rho(x_0), \ldots, \rho(x_{a-1}))$ is equal to $0$. Choose $n$ large enough so that $n > \Delta(x_i, x_{i+1})$ for all $i < a-1$. Then whenever $x \restriction_n = \rho(x_0) \restriction_n$, we get that the volume of $(\rho(x), \rho(x_1), \ldots, \rho(x_{a-1}))$ is equal to $0$. Hence $\{\rho(x): x \in C, x \restriction_n = x_0 \restriction_n\}$ is contained in an $a-2$-dimensional hyperplane; hence whenever $u = \{y_i: i \leq a\} \in \binom{C}{a}$ is such that each $y_i \restriction_n = \rho(x_0) \restriction_n$, we get that $\rho[u]$ has volume $0$. Since every element of $\binom{C}{a}$ has the same type as some such $u$, we get that for all $u \in \binom{C}{a}$, $\rho[u]$ has volume $0$, and so $\rho[C]$ is contained in an $a-2$-dimensional hyperplane.

Let $a$ be largest so that this fails; thus $\rho[C]$ is a subset of an $a-1$-dimensional hyperplane, but for each $a' \leq a$ and for each $u \in \binom{C}{a}$, $\rho[u]$ has nonzero volume. We claim that for each $a' \leq a$, $\rho[C]$ is $a'$-rainbow (and hence strongly $a$'-rainbow).

Suppose not, say $a' \leq a$ and $x_0 <_{lex} \ldots <_{lex} x_{a'-1}$ and $y_{0} <_{lex} \ldots <_{lex} y_{a'-1}$ witness this, so $\{x_i: i < a'\} \not= \{y_i: i <a'\}$ and yet their images under $\rho$ have the same $a'-1$-ary volume. Let $N$ be large enough such that $N > \Delta(x_i, x_j)$ and $N > \Delta(y_i, y_j)$ for all $i < j < a'$, and whenever $x_i \not= y_j$ then $N > \Delta(x_i, y_j)$. Now choose $i_* < a'$ such that $x_{i_*} \not \in \{y_i: i < a'\}$; then for any $x$ with $x \restriction_N = x_{i_*} \restriction_N$ we have that the $a'$-ary volume of $\rho[x_0, \ldots, x_{i_*-1}, x_{i_*}, x_{i_{*}+1}, \ldots, x_{a'-1}]$ is equal to the $a'-1$-ary volume of $\rho[x_0, \ldots, x_{i_*-1}, x, x_{i_*+1}, \ldots, x_{a'-1}]$, both being equal to the $a'-1$-ary volume of $y_0, \ldots, y_{a'-1}$. 

 Given $\overline{z}$ with $f_*(\overline{z}) = f_*(\overline{x})$, let $N_{\overline{z}}$ be the maximum of $\Delta(z_i, z_j)+1: i < j < a'$. Let the cone above $\overline{z}$, $\mathcal{C}_{\overline{z}}$, be all $z$ such that $z \restriction_{N_{\overline{z}}} = z_{i_*} \restriction_{N_{\overline{z}}}$. Then for any $z \in \mathcal{C}_{\overline{z}}$, $\rho[\overline{z}]$ has the same $a'$-ary volume as $\rho[\overline{z} \backslash \{z_{i_*}\} \cup \{z\}]$. 

Recall that $q_{a'}(v_0, \ldots, v_{a'-1}, w_0, \ldots, w_{a'-1})$ is a polynomial (each $v_i, w_j$ is a tuple of $d$-variables) such that given $\alpha_0, \ldots, \alpha_{a'-1}, \beta_0, \ldots, \beta_{a'-1}$ from $\mathbb{R}^d$, $\overline{\alpha}$ and $\overline{\beta}$ have the same $a'-1$-ary volume iff $q_{a'}(\overline{\alpha}, \overline{\beta}) = 0$; given tuples $\overline{\alpha}, \overline{\beta}$ from $\mathbb{C}^d$, then we define them to have the same $a'-1$-ary volume if $q_{a'}(\overline{\alpha}, \overline{\beta})= 0$.

Inductively choose $\overline{x}^n: n < \omega$ so that each $\overline{x}^n = x^n_0 <_{lex} \ldots <_{lex} x^n_{a'-1}$ has $f_*(\overline{x}^n) = f_*(\overline{x})$, and each $\overline{x}^{n+1} \subseteq \mathcal{C}_{\overline{x}_n}$. For each $n$, let $V_n$ be the set of all $\alpha \in \mathbb{C}^d$ such that for each $m \leq n$, the $a'-1$-ary volume of $\rho[\overline{x}^m]$ is equal to the $a'-1$-ary volume of $\rho[\overline{x}^m \backslash \{x^m_{i_*}\}] \cup \{\alpha\}$. This is a descending chain of prevarieties, so to get a contradiction it suffices to show that $V_{n+1} \subsetneq V_n$. But choose $i \not= i_*$; then $x^{n+1}_i \in V_n$, but  $\rho[\overline{x}^{n+1} \backslash \{x^{n+1}_{i_*}\}] \cup \{\rho(x^{n+1}_{i})\}$ is a degenerate simplex, so has $a'-1$-ary volume zero, so $\rho(x^{n+1}_i) \not \in V_{n+1}$. 
\end{proof}

From Theorem~\ref{th:perfect} and the comments proceeding it we
obtain the following:

\begin{corollary}\label{co:darling}~
\begin{enumerate}
\item
Suppose $X \subseteq \mathbb{R}^d$ is analytic and uncountable.
Then there is a perfect set $Q \subseteq X$ and some 
$2 \leq a \leq d+1$ such that $Q$ is strictly $a$-rainbow.
\item
Assume sufficient large cardinals.
Suppose $X \subseteq \mathbb{R}^d$ is projective and uncountable.
Then there is a perfect set $Q \subseteq X$ and some 
$2 \leq a \leq d+1$ such that $Q$ is strictly $a$-rainbow.
\item
Assume PSP.
Suppose $X \subseteq \mathbb{R}^d$ is uncountable.
Then there is a perfect set $Q \subseteq X$ and some 
$2 \leq a \leq d+1$ such that $Q$ is strictly $a$-rainbow.
\end{enumerate}
\end{corollary}

\section{A model of set theory where an uncountable set of reals has no uncountable $2$-rainbow subset}\label{independence}

In this section we prove it is consistent with $ZF$ that there is an uncountable set of reals without an uncountable $2$-rainbow subset. Thus some amount of choice is necessary. The proof is a standard symmetric models argument; for a source on this, see \cite{Jech}, Chapter 15. 

We attempted to prove consistency over $ZF + DC$, but could not, so we leave the following as an open question:

\vspace{2 mm}

\noindent \textbf{Question.} Is it consistent with $ZF + DC$ that there is an uncountable set of reals without an uncountable $2$-rainbow subset?

\begin{theorem}
Suppose $\mathbb{V} \models ZFC$. Then there is a symmetric submodel $M$ of a forcing extension $\mathbb{V}[G]$ of $\mathbb{V}$, such that $M \models ZF + $ there is an uncountable set of reals with no uncountable $2$-rainbow subset. 
\end{theorem}
\begin{proof}
We identify $x \in 2^\omega$ with the element of $[0, 1]$ with binary expansion given by $x$. (The collisions do not matter.)

Let $P$ be the forcing notion of all finite partial functions from $\omega \times \omega \to 2$. Then forcing by $P$ adds a Cohen-generic $\dot{f} \in (2^{\omega \times \omega})$.

For each $n < \omega$, let $\dot{a}_n$ be a $P$-name for $\{m < \omega: \dot{f}(n, m) = 1\}$, so each $0_P \Vdash \dot{a}_n \subseteq \omega$. For each $s \subset \omega$ finite let $\dot{a}_{n, s}$ be a $P$-name for $\dot{a}_n \Delta s$ (the symmetric difference). Let $\dot{A}$ be a $P$-name for $\{\dot{a}_{n, s}: n < \omega, s \subset \omega \mbox{ finite}\}$.

Let $\mathcal{G}$ be the group of all permutations $\sigma$ of $\omega \times \omega \times 2$ such that: there is a permutation $\sigma_{0}$ of $\omega$, such that for all $(n, m, i)$, $\sigma(n, m, i) = (\sigma_0(n), m, j)$ for some $j$ (thus we have a map $\sigma \mapsto \sigma_0$). Each $\sigma \in \mathcal{G}$ induces an automorphism of $P$ and hence of its Boolean completion $\mathcal{B}$, which we identify with $\sigma$. Let $\mathcal{F}$ be the filter of subgroups on $\mathcal{G}$, generated by $\mbox{Fix}(\dot{a}_n)$ for each $n < \omega$, along with $\mbox{Fix}(\dot{A})$. Note that $\sigma \in \mbox{Fix}(\dot{a}_n)$ iff $\sigma \restriction_{\{n\} \times \omega \times 2}$ is the identity, and $\sigma \in \mbox{Fix}(\dot{A})$ iff for each $n < \omega$, there are only finitely many $m$ with $\sigma(n, m, 0) \not= (\sigma_0(n), m, 0)$. 

Let $\mathbb{V}[G]$ be a forcing extension by $P$, and let $M$ be the symmetric submodel determined by $\mathcal{F}, \mathcal{G}$, that is $M$ is the set of all $\dot{a}^{\mathbb{V}[G]}$, for $\dot{a}$ a $P$-name such that $\mbox{Fix}(\dot{a}) \in \mathcal{F}$ (as computed in $\mathbb{V}$), and moreover this holds hereditarily. $M$ is a model of $ZF$; see Chapter 15 of \cite{Jech}. Now $\dot{A}^{\mathbb{V}[G]} \in M$ since $\mbox{Fix}(\dot{A}) \in \mathcal{F}$ by construction and this also holds for each $\dot{a} \in \dot{A}$. We claim that $M, \dot{A}^{\mathbb{V}[G]}$ works.

It is well known that $\dot{A}^{\mathbb{V}[G]}$ is uncountable in $\mathbb{V}[G]$: if we look at the larger model $N := \mathbb{V}(\{\dot{a}_n^{\mathbb{V}[G]}: n \in \omega)$ then this is the standard example, due to Cohen, of a model of ZF where choice fails; $\{\dot{a}_n^{\mathbb{V}[G]}: n \in \omega\}$ is an uncountable set and in fact it has no countable subset. See Chapter 14 of \cite{Jech}. So the larger set $\dot{A}$ must still be uncountable in the smaller model $M$.

Suppose towards a contradiction that in $M$, $\dot{A}^{\mathbb{V}[G]}$ had an uncountable $2$-rainbow subset. Then we can choose some hereditarily $\mathcal{F}$-symmetric $P$-name $\dot{B}$ and some $p \in P$ such that $p$ forces: $\dot{B}$ is an uncountable $2$-rainbow subset of $\dot{A}$. We can choose $N$ large enough so that $\mbox{dom}(p) \subset N \times \omega$ and, for every $\sigma \in \mathcal{G}$, if $\sigma \restriction_{N \times \omega \times 2}$ is the identity and if $\sigma \in \mbox{Fix}(\dot{A})$, then $\sigma(\dot{B}) = \dot{B}$.

For each $n < \omega$ and $s \subset \omega$ finite let $Q_{n}$ be the set of all $q \leq p$ such that $q$ forces: $\dot{B} \cap \{\dot{a}_{n, s}: s \subset \omega \mbox{ finite}\} \not= \emptyset$. 

We claim that each $Q_n$ is nonempty, i.e. $p$ does not force that $\dot{B} \cap \{\dot{a}_{n, s}: s \subset \omega \mbox{ finite}\}$ is empty. For suppose it did; then for every $n' \geq N$, $p$ forces that $\dot{B} \cap \{\dot{a}_{n', s}: s \subset \omega \mbox{ finite}\}$ is empty (by considering $\sigma \in \mathcal{G}$ that fix the second and third coordinates and interchange $n, n'$). But then $p$ would force that $\dot{B} \subseteq \{\dot{a}_{n, s}: n < N, s \subset \omega \mbox{ finite}\}$, a countable set.

We claim that for each $n \geq N$, and for each $q \in Q_n$, we have that $q \restriction_{N \times \omega \times 2} \in Q_n$. Indeed, given some $q'$ such that $q' \restriction_{N \times \omega\times 2} = q \restriction_{N \times \omega \times 2}$ it is not hard to find some $\sigma \in \mbox{Fix}(\dot{A})$ with $\sigma_0$ the identity and with $\sigma \restriction_{N \times \omega \times 2}$ the identity, and with $\sigma(q)$ compatible with $q'$. Since $\sigma$ fixes $\dot{B} \cap \{\dot{a}_{n, s}: s \subset \omega \mbox{ finite}\}$, and since $q$ forces this set to be nonempty, $\sigma(q)$ does as well; hence so does $\sigma(q) \cup q'$. We have shown that $Q_n$ is dense below $q \restriction_{N \times \omega}$; hence $q \restriction_{N \times \omega} \in Q_n$.

Thus we can choose $q \in Q_N$ with support contained in $N \times \omega$. (By symmetry again, we see that actually $q \in Q_n$ for all $n \geq N$, from which it follows that $p \in Q_n$ for all $n$, but we won't need this.) For each $s \subset \omega$ finite let $R_{N, s}$ be the set of all $r \leq q$ such that $r$ forces $\dot{a}_{N, s} \in \dot{B}$. For each $s \subset \omega$ finite let $\sigma^s \in \mathcal{G}$ be the permutation defined by: $\sigma^s(n', m, i) = (n',m, i)$ unless $n' = n$ and $m \in s$, in which case $\sigma^s(n', m, i) = (n', m, 1-i)$. Note that if $r \in R_{N, s}$ then for all $t \subset \omega$ finite, $\sigma^t(r) \in R_{N, t \triangle s}$ (because $\sigma^t(\dot{a}_{N, s})  = \dot{a}_{N, t \triangle s}$).

Thus, since some $R_{N, s}$ must be nonempty, we get that they all must be nonempty. Choose $r \in R_{N, \emptyset}$. Write $r = r_0 \sqcup r_1$ where $r_1 = r \restriction_{\{N\} \times \omega \times 2}$. Choose $N'$ large enough so that $\mbox{dom}(r_1) \subseteq \{N\} \times N' \times 2$. Then for every $s \subset \omega$ finite with $s \cap N' = \emptyset$, we have that $\sigma^s(r) = r \in R_{N, s}$. Thus for every $s \subset \omega$ finite with $s \cap N' = \emptyset$, we have that $r \Vdash \dot{a}_{N, s} \in \dot{B}$. But this is a contradiction: let $s_0 = \emptyset$, let $s_1 = \{N'\}$, let $t_0 = \{N'+1\}$, let $t_1 = \{N', N'+1\}$. $r$ forces that each $s_i, t_j \in \dot{B}$, but $0_P$ forces that the distance from $\dot{a}_{N,s_0}$ to $\dot{a}_{N,s_1}$ is $\frac{1}{2^{N'+1}}$, as is the distance from $\dot{a}_{N,t_0}$ to $\dot{a}_{N,t_1}$.
\end{proof}

\section{Acknowledgements}

We want to thank Carolyn Gasarch whose skepticism of the axiom of choice
was one of our inspirations. We also want to thank our co-authors on
the original {\it Distinct Volume Sets} paper,
(David Conlon, Jacob Fox, David Harris, and Sam Zbarsky) since that paper
was also one of our inspirations. Finally, we would like to thank Chris Laskowski for greatly improving the exposition of Theorem~\ref{IndiscSingularUncountCof}.

%\bibliographystyle{abbrv}
%\bibliography{bibfile}

\end{document}